\documentclass[12pt]{article}
\pdfoutput=1
\usepackage[selected=true,shrink=15,stretch=30,letterspace=30]{microtype}
\usepackage{amsthm}
\usepackage{amsmath}
\usepackage{amssymb}
\usepackage{xcolor,paralist}
\usepackage[hyphens]{url}
\usepackage[unicode=true,pdftex,colorlinks=true, pdfborderstyle={/S/D/D[2 1]/W .5},linkcolor=blue,citecolor=violet,urlcolor=black,pdfversion=1.5,pdflang=en-GB
,pdfauthor={Alexander Dyachenko and Dmitrii Karp},pdfsubject={MSC(2010) 33C05, 30B70, 47B50}
,pdftitle={Ratios of the Gauss hypergeometric functions with parameters shifted by integers: more on integral representations}
,pdfkeywords={Gauss hypergeometric function, Integral representation, Runckel's theorem}
]{hyperref}

\newcommand{\coloneqq}{\mathrel{\mathop:}=}
\newcommand{\eqqcolon}{=\mathrel{\mathop:}}
\newcommand{\sign}{\operatorname{sign}}

\def\R{\mathbb{R}}
\def\C{\mathbb{C}}
\def\N{\mathbb{N}}
\def\Z{\mathbb{Z}}
\newtheorem{theorem}{\hspace*{\parindent}Theorem}
\newtheorem{lemma}{\hspace*{\parindent}Lemma}
\newtheorem{corollary}{\hspace*{\parindent}Corollary}
\title{Ratios of the Gauss hypergeometric functions with parameters shifted by integers:
    more on integral representations}
\author{A.\:Dyachenko$^{\rm a}$\footnote{Corresponding author. E-mail: A.\:Dyachenko -- \emph{diachenko@sfedu.ru},
D.\:Karp --  \emph{dimkrp@gmail.com}}~~and D.\:Karp$^{\rm b}$
\\[10pt]
\\
\small{\textit{$\phantom{1}^a$Keldysh Institute of Applied Mathematics (RAS), Miusskaya sq.~4, 125047 Moscow,  Russia}}
\\
\small{\textit{$\phantom{1}^b$Far Eastern Federal University, Ajax Bay~10, 690922 Vladivostok, Russia}}
}
\date{}
\begin{document}
\maketitle


\begin{abstract}
We consider the ratio of two Gauss hypergeometric functions, in which the parameters of the
    numerator function differ from the respective parameters of the denominator function by
    integers. We derive explicit integral representations for this ratio based on a formula for
    its imaginary part. This work extends our recent results by lifting certain restrictions on
    parameters. The new representations are illustrated with a few examples and an application
    to products of ratios.
\end{abstract}

\bigskip


\bigskip


\bigskip

\section{Introduction}

Let~$(a)_n\coloneqq a(a+1)\cdots(a+n-1)$, ~$(a)_0=1$, denote the rising factorial. The Gauss
hypergeometric function (\cite[Chapter~II]{HTF1}, \cite[Chapter~15]{DLMF}, \cite{Gauss}) is
defined as the analytic continuation of the sum of the power series
\begin{equation}\label{eq:2F1_def}
    {}_2F_1(a,b;c;z)={}_2F_1\bigg(\!
    {\renewcommand{\arraystretch}{1.1}
        \begin{array}{l}a,b\\c
        \end{array}}
    \bigg|\,\,z\bigg)=\sum\limits_{n=0}^{\infty}\frac{(a)_n(b)_n}{(c)_nn!}z^n.
\end{equation}
For~$a,b\notin-\N_0$ (if this condition is violated, then ${}_2F_1$
reduces to a polynomial), one usually introduces the branch cut~$[1,+\infty)$ to make it
analytic and single-valued in the rest of the complex plane. The functions~${}_2F_1(a,b;c;z)$
and ${ }_2F_1(a+n_1,b+n_2;c+m;z)$, ~$n_1,n_2,m\in\Z$, are called associated~\cite[p.~58]{HTF1}
or contiguous in a wide sense. Gauss showed that any three functions of this type satisfy a
linear relation with coefficients rational in $a,b,c,z$. For $n_1,n_2,m\in\{-1,0,1\}$, this
relation has coefficients linear in $z$, and the functions are called contiguous (in a narrow
sense). In our recent paper \cite{DK} we initiated a study of the ratios
\begin{equation}\label{eq:gen-ratio-def}
    R_{n_1,n_2,m}(z)=\frac{{}_2F_1(a+n_1,b+n_2;c+m;z)}{{}_2F_1(a,b;c;z)},
\end{equation}
with arbitrary integer ~$n_1,n_2,m$. The particular case $R_{0,1,1}$ was investigated already by
Gauss who found a continued fraction for this ratio which, under additional restrictions on
parameters, becomes a Stieltjes or $S$-fraction convergent to a Stieltjes transform of a
positive measure. Explicit density of this measure was found many decades later by Belevitch
\cite{Bel}. In~\cite{DK}, we extended this integral representation to the general
$R_{n_1,n_2,m}(z)$ under the assumptions that~$a,b,c\in\R$, ~$c,c+m\notin-\N_0$
and~$R_{n_1,n_2,m}(z)$ has no poles in~$\C\setminus[1,+\infty)$. We additionally assumed that
the behaviour of $R_{n_1,n_2,m}(z)$ near~$z=1$ is mild, so that the singularity at this point is
integrable. The main purpose of the present paper is to drop these restrictions and extend the
representations given in~\cite{DK} to a more general setting. Firstly, we will get rid of the
assumption of integrability near the point $z=1$ and allow arbitrary behaviour in the
neighbourhood of this point. Secondly, we will remove the assumption that $R_{n_1,n_2,m}(z)$ is
free of poles in~$\C\setminus[1,+\infty)$. However, the representation we obtain for ratios with
poles will depend on rational functions whose numerators and denominators have explicit degrees
but unknown coefficients. Calculation of these coefficients requires the knowledge of the zeros
of the Gauss hypergeometric function and residues of $R_{n_1,n_2,m}(z)$ at these zeros (for
generic values of its parameters). These, of course, cannot be given explicitly, but may, in
principle, be computed numerically.

\section{Asymptotic behaviour and boundary values}

It will be convenient to use the following notation: if~$a$ is a real number, then
\[
(a)_-\coloneqq \min(a,0)
\quad\text{and}\quad
(a)_+\coloneqq \max(a,0).
\]
Derivation of the integral representations for the ratio~$R_{n_1,n_2,m}$ will require certain
estimates of its asymptotic behaviour. The following is a condensed and corrected form of
\cite[Subsection~2.1]{DK}:
\begin{lemma}\label{lm:asymptotic}
    Let~$c,c+m\notin-\N_0$, and let~$a,b\in\R$. Then there exist four
    constants~$\varepsilon_1,\varepsilon_\infty\in\{-1,0,1\}$ and~$L_1,L_\infty\ne 0$
    independent of~$z$ such that
    \begin{align}\label{eq:asymp1}
      R_{n_1,n_2,m}(z)
      &=L_1\,(1-z)^{\eta(a+n_1,b+n_2,c+m)-\eta(a,b,c)}
        \,[\log(1-z)]^{\varepsilon_1}\big(1+o(1)\big)
        &\text{as}~~z\to1;
      \\[2pt]
      \label{eq:asymp-inf}
      R_{n_1,n_2,m}(-z)
      &=L_\infty\, z^{\zeta(a+n_1,b+n_2,c+m)-\zeta(a,b,c)}
        \,[\log(z)]^{\varepsilon_\infty}\big(1+o(1)\big)
        &\text{as}~~z\to\infty,
    \end{align}
    where we put
    \begin{equation} \label{eq:asymp_1_eta}
        \eta(a,b,c)=\begin{cases}
            (c-a-b)_{+},& \text{if~~}-a,b-c\in\N_0\text{~~or~~}-b,a-c\in\N_0;\\
            0,& \text{if } -a\in\N_0\text{ and/or }-b\in\N_0\text{ while }a-c,b-c\notin\N_0;\\
            c-a-b,& \text{if } -a,-b\notin\N_0,\text{ while }a-c\in\N_0\text{ and/or }b-c\in\N_0;\\
            (c-a-b)_{-},&\text{otherwise}
        \end{cases}
    \end{equation}
    and
    \begin{equation} \label{eq:asymp_inf_eta}
        \zeta(a,b,c)=\begin{cases}
            -a,& \text{if~~} b-c\in\N_0 \text{~~and/or~~} -a\in\N_0;\\
            -b,& \text{if~~} a-c\in\N_0 \text{~~and/or~~} -b\in\N_0;\\
            -\min(a,b),&\text{otherwise.}
        \end{cases}
    \end{equation}
\end{lemma}

\noindent(Note that the above formulae~\eqref{eq:asymp1}--\eqref{eq:asymp-inf} also work for the
degenerate cases --- i.e. when~${}_2F_1(a,b;c;z)$ is a polynomial or polynomial multiple of power of $1-z$.)

Another important ingredient is the next theorem giving an explicit representation for the
imaginary part of $R_{n_1,n_2,m}(z)$ on the banks of the branch cut $[1,+\infty)$.
Given~$n_1,n_2,m\in\Z$ let us introduce the following related quantities:
\begin{equation}\label{eq:nmrelated}
\begin{split}
\underline{n}=\min(n_1,n_2),&~~\overline{n}=\max(n_1,n_2),~~ p=(m-n_1-n_2)_{+},~~  l=(n_1+n_2-m)_{+},
\\
&r=l+(m)_{+}-\underline{n}-1=\left\{\!\!\!
\begin{array}{l}
\max(m-\underline{n},\overline{n})-1,~~m\ge0
\\
\max(-\underline{n},\overline{n}-m)-1,~~m\le0.
\end{array}
\right.
\end{split}
\end{equation}
Observe that $p-l=m-n_{1}-n_{2}$ and that $r$ is only negative when $n_1=n_2=m=0$, in which
case $r=-1$.

\begin{theorem}[{\cite[Theorem~2.11]{DK}}]\label{th:2F1ratioboundary}
    Suppose that~$n_1,n_2,m\in\Z$. On the banks of the branch cut $x>1$, the following expression holds
    \begin{subequations}\label{eq:2F1ratioboundary_with_B}
        \begin{equation}\label{eq:2F1ratioboundary}
            \Im[R_{n_1,n_2,m}(x\pm i0)]
            =\pm{\pi}B_{n_1,n_2,m}(a,b,c)\frac{x^{l-\underline{n}-c}(x-1)^{c-a-b-l}P_{r}(1/x)}{|{}_{2}F_{1}(a,b;c;x)|^{2}},
        \end{equation}
        where
        \begin{equation}\label{eq:B-defined}
            B_{n_1,n_2,m}(a,b,c)
            =-\frac{\Gamma(c)\Gamma(c+m)}{\Gamma(a)\Gamma(b)\Gamma(c-a+m-n_1)\Gamma(c-b+m-n_2)}
        \end{equation}
    \end{subequations}
    and $P_r(t)$ is a polynomial of degree $r$ ~\emph{(}$P_{-1}\equiv0$\emph{)}.
\end{theorem}
Note that the coefficients of~$P_r$ depend on the parameters, so the whole
expression~\eqref{eq:2F1ratioboundary} may remain nonzero even when~$B_{n_1,n_2,m}(a,b,c)$
vanishes. The polynomial~$P_r(t)$ can be computed via the Taylor expansion of the underlying
hypergeometric identity~\cite[eq.~(2.19)]{DK} multiplied by~$t^{-\underline{n}}(1-t)^{p}$, in
which $\alpha=a$, $\beta=1-c+a$, $\gamma=1-b+a$. Alternatively, one can write~$P_r(t)$
explicitly in the following form:
\begin{lemma}[{\cite[Lemma~2]{CKP}}]\label{lm:2F1identity}
    For~$r\ge0$, the polynomial $P_r(t)$ in~\eqref{eq:2F1ratioboundary} is given by
    \begin{subequations}\label{eq:Prexplicit}
        \begin{equation*}
            P_{r}(t)=(-1)^{\overline{n}}\sum_{k=0}^{r}(-t)^k
            \!\!\sum\limits_{j=(k-p)_{+}-\overline{n}}^{k-\overline{n}}\!\!\!\!\!\!
            (-1)^{j}\binom{p}{k-\overline{n}-j}K_{j},
        \end{equation*}
        where, with the convention $1/(-i)!=0$ for $i\in\N$,
\begin{multline*}
K_{j}=\frac{(1-a)_{j}(c-a)_{m+j}}{(b-a)_{n_2+j+1}(j+n_1)!}
{}_{4}F_{3}\!\left(\!\left.\!\begin{array}{l}-j-n_1,a,1+a-c,a-b-n_2-j
\\
a-j,1+a-c-m-j,1+a-b\end{array}\!\vline\,\,1\right.\right)
\\
+
\frac{(1-b)_{j}(c-b)_{m+j}}{(a-b)_{n_{1}+j+1}(j+n_2)!}
{}_{4}F_{3}\!\left(\!\left.\!\begin{array}{l}-j-n_2,b,1+b-c,b-a-n_{1}-j
\\
b-j,1+b-c-m-j,1+b-a\end{array}\!\vline\,\,1\right.\right)\!.
\end{multline*}
     \end{subequations}
\end{lemma}

At least one of the numbers~$a,b,c-a,c-b$ lies in $-\N_0$ exactly when~${ }_2F_1(a,b;c;z)$
degenerates to a polynomial, possibly times a fractional power of~$(1-z)$. In this case the
denominator in~\eqref{eq:2F1ratioboundary} in Theorem~\ref{th:2F1ratioboundary} may vanish for a
certain~$x>1$: for instance,
\[
    { }_2F_1\left(1,-2;\frac 45;\frac 65\right)
    ={ }_2F_1\left(-\frac 15,\frac{14}5;\frac 45;\frac 65\pm i0\right)=0.
\]
Runckel showed that this situation is impossible in the non-degenerate case. More specifically,
the following fact is a direct corollary of \cite[Lemma 2]{Runckel}:
\begin{lemma}\label{lemma:no-zeros-branch-cut}
    If~$a,b,c-a,c-b\notin-\N_0$ and~$x>1$, then~${ }_2F_1(a,b;c;x\pm i0)\ne 0$.
\end{lemma}
\begin{proof}
    The case~$c-a-b\ge 0$ and~$b\ge a$ is the first assertion of~\cite[Lemma 2]{Runckel}. Due to
    the symmetry with respect to exchanging~$a\leftrightarrow b$, the lemma also holds for the
    case~$c-a-b\ge 0$ and~$b\le a$. Now, if~$c-a-b\le 0$ and~$x>1$, the right hand side of
    Euler's identity
    \begin{equation} \label{eq:Euler_ID}
        {}_2F_1(a,b;c;z)=(1-z)^{c-a-b}{}_2F_1(c-a,c-b;c;z)
    \end{equation}
    does not vanish for~$z=x\pm i0$, and hence neither does the left hand side when~$x>1$.
\end{proof}

\section{Integral representations}\label{sec:boundary-values-n-integrals}

Our goal is to derive explicit integral representations for the ratio~$R_{n_1,n_2,m}(z)$. This
ratio is known to have at most finitely many poles in the cut plane~$\C\setminus[1,+\infty)$ and
on both banks of the branch cut~$(1,+\infty)$, see Theorem~\ref{th:2F1zeros} below. It may also
have at most a polynomial growth near the branch points~$z=1$ and~$z=\infty$. There are two
options for dealing with the poles and the growth: to subtract the corresponding rational
correction term, or to multiply~$R_{n_1,n_2,m}(z)$ by a specially tailored polynomial. We are
going to deal with both options via an adapted version of the Schwarz formula stated here as
Lemma~\ref{lemma:Schwarz_asympt}.

A function meromorphic in the (open) upper half of the complex plane is called real if it
extends as a meromorphic function to the (open) lower half of the complex plane according to the
rule:
\[
f(\overline z) = \overline{f(z)}
\]
wherever~$f(z)$ is defined. If~$f(z)$ is analytic near the origin, then~$f(z)$ is real precisely
when all coefficients of its Taylor expansion at the origin are real.

\begin{lemma}\label{lemma:Schwarz_asympt}
    Let~$f(z)$ be a real function meromorphic in the cut plane~$\C\setminus[1,+\infty)$ and
    analytic near the origin. Suppose that there exists a real polynomial~$q(z)$ of degree~$M$,
    for which~$q(z)f(z)$ is analytic in~$\C\setminus[1,+\infty)$ and~$q(x) u(x)$ is continuous
    for $x\in(1,+\infty)$, where~$u(x)\coloneqq\frac 1\pi \Im f(x + i0)$. If~$N\in\N_0$ is such
    that
    \begin{equation}\label{eq:Schwarz_Qasympt}
        \lim_{|z-1|\to 1}\big|q(z)(1-z)f(z)\big|=\lim_{|z|\to\infty}\big|f(z)z^{M-N}\big|=0
    \end{equation}
    and~$q(x)u(x)/x^{N+1}$ is absolutely integrable on~$(1,+\infty)$, then
    \begin{equation}\label{eq:Schwarz_Qplus}
        f(z)
        = \sum_{k=0}^{N-1}\frac{z^k}{q(z)}\sum_{j=0}^{k}\frac{q^{(k-j)}(0)f^{(j)}(0)}{(k-j)!j!}
        + \frac{z^N}{q(z)} \int_{1}^{+\infty}\frac{q(x)u(x)\,dx}{(x-z)x^N}.
    \end{equation}
\end{lemma}
\begin{proof}
    Let~$C$ be the closed contour consisting of a small circle around the point~$z=1$ of
    radius~$\epsilon<1/2$, then (the upper bank of) the interval~$(1+\epsilon+i0,1/\epsilon+i0)$
    followed by a large circle~$|z|=1/\epsilon$ and (the lower bank of) the
    interval~$(1+\epsilon-i0,1/\epsilon-i0)$. The contour is traversed so that the bounded
    domain inside it is on the left (in particular, the large circle is traversed in the
    anticlockwise direction). Given~$k\in\N_0$, the Cauchy formula for the $(N+k)$th Taylor
    coefficient of~$g(z)\coloneqq q(z)f(z)$ reads
    \begin{equation} \label{eq:Cauchy_for_Schwarz}
        \begin{aligned}
            2\pi i\frac{g^{(N+k)}(0)}{(N+k)!}
            &= \oint_C \frac{g(z)}{z^{N+k+1}} dz
            \\
            &=\oint_{|z-1|=\epsilon} \frac{g(z) \,dz}{z^{N+k+1}}
            +\int_{1+\epsilon}^{1/\epsilon} \frac{g(x+i0)-g(x-i0)}{x^{N+k+1}}dx
            +\oint_{|z|=1/\epsilon} \frac{g(z) \,dz}{z^{N+k+1}}
            ,
        \end{aligned}
    \end{equation}
    note also that
    \[
        g(x+i0)-g(x-i0)=2i\Im g(x+i0) = 2\pi i q(x)u(x).
    \]
    On letting~$\epsilon\to+0$, the first and the last integrals on the right-hand side
    of~\eqref{eq:Cauchy_for_Schwarz} vanish due to~\eqref{eq:Schwarz_Qasympt}. Hence,
    \[
        \frac{g^{(N+k)}(0)}{(N+k)!}
        =
        \frac1{\pi}\int_{1}^{+\infty} \frac{\Im g(x+i0)}{x^{N+k+1}}dx
        =
        \int_{1}^{+\infty} \frac{q(x)u(x)\,dx}{x^{N+k+1}}
        \eqqcolon c_k
        ,\quad k\in\N_0.
    \]
    The power series~$\sum_{k=0}^\infty c_k z^k$ uniformly converges on compact subsets of the
    unit disc (in fact,~$c_k\to 0$ as~$k\to\infty$ since for each~$x>1$ the integrand
    monotonically tends to zero). Therefore, if~$|z|<1$ we have
    \begin{equation}\label{eq:Schwarz_series}
        \sum_{k=0}^{\infty}c_kz^k
        = \int_{1}^{+\infty}\left(
            \sum_{k=0}^{\infty}\frac{z^k}{x^{k+1}}\right)\frac{q(x)u(x)\,dx}{x^N}
        =
        \int_{1}^{+\infty}\frac{q(x)u(x)\,dx}{(x-z)x^N}.
    \end{equation}
    A comparison between the left-hand side and the Taylor expansion of~$g(z)$ at the origin
    shows that
    \begin{equation*}
        z^N\sum_{k=0}^{\infty}c_kz^k
        =
        g(z)-\sum_{k=0}^{N-1}\frac{z^k}{k!}g^{(k)}(0)
    \end{equation*}
    for~$z$ inside the unit disc. At the same time, the ratio~$x/(x-z)$ is bounded in~$z$ on
    compact subsets of~$\C\setminus[1,+\infty)$ uniformly in~$x>1$, so the integral on the
    right-hand side of~\eqref{eq:Schwarz_series} uniformly converges there to an analytic
    function. Thus, for all~$z\in\C\setminus[1,+\infty)$
    \begin{equation*}
        g(z)
        =\sum_{k=0}^{N-1}\frac{z^k}{k!}g^{(k)}(0)
        + z^N\int_{1}^{+\infty}\frac{q(x)u(x)\,dx}{(x-z)x^N}
        .
    \end{equation*}
    This expression yields~\eqref{eq:Schwarz_Qplus} after division by~$q(z)$ and the
    substitutions~$g(z)=q(z)f(z)$ and
    \[
        g^{(k)}(0)
        =\sum_{j=0}^{k}\binom{k}{j}q^{(k-j)}(0)f^{(j)}(0)
        .
    \]
\end{proof}

\begin{corollary}\label{cr:Schwarz_asympt}
    Let~$f(z)$ be a real analytic function defined in the cut plane~$\C\setminus[1,+\infty)$
    such that $u(x)\coloneqq\frac 1\pi \Im f(x + i0)$ is continuous on $(1,+\infty)$. Suppose
    also that there exist~$M,N\in\N_0$ for which
    \begin{equation}\label{eq:Schwarz_asympt}
        \lim_{|z-1|\to 1}\big|f(z)(1-z)^{M+1}\big|=\lim_{|z|\to\infty}\big|f(z)z^{M-N}\big|=0
    \end{equation}
    and~$u(x)(x-1)^{M}/x^{N+1}$ is absolutely integrable over~$(1,+\infty)$. Then
    \begin{equation}\label{eq:Schwarz_plus}
        f(z)
        = \sum_{k=0}^{N-1}
        \frac{z^k}{(1-z)^M}\!\sum_{j={\max\{k-M,0\}}}^k\!(-1)^{k-j}\binom{M}{k-j}\frac{f^{(j)}(0)}{j!}
        + \frac{z^N}{(1-z)^M} \int_{1}^{+\infty}\frac{(1-x)^M \,u(x)\,dx}{(x-z)x^N}.
    \end{equation}
\end{corollary}
\begin{proof}
    Put~$q(z)\coloneqq (1-z)^M$ in Lemma~\ref{lemma:Schwarz_asympt}. Then, due to
    \begin{equation*}
      \frac{z^k}{q(z)}\sum_{j=0}^{k}q^{(k-j)}(0)\frac{f^{(j)}(0)}{(k-j)!j!}
      =\frac{z^k}{(1-z)^M}
      \!\sum_{j={\max\{k-M,0\}}}^k\!\frac{M!(-1)^{k-j}}{(M-(k-j))!}\cdot\frac{f^{(j)}(0)}{(k-j)!j!}
      ,
  \end{equation*}
  the expression in ~\eqref{eq:Schwarz_Qplus} becomes~\eqref{eq:Schwarz_plus}.
\end{proof}

One can also reshape the sum in  formula~\eqref{eq:Schwarz_plus} as follows:
\[
    \sum_{k=0}^{N-1} \frac{z^k}{(1-z)^M}
    \!\sum_{j={\max\{k-M,0\}}}^k\!(-1)^{k-j}\binom{M}{k-j}\frac{f^{(j)}(0)}{j!}
    =\sum_{j=0}^{N-1}\frac{f^{(j)}(0)}{j!}
    \sum_{k=0}^{\!\min\{N-j-1,M\}\!}\!\binom{M}{k}\frac{(-1)^{k}z^{k+j}}{(1-z)^M}
    .
\]

\begin{theorem}\label{th:2F1ratio-gen-repr}
    Suppose that~$a,b,c\in\R$ and~$n_1,n_2,m\in\Z$, where~$-c,-c-m\notin\N_0$. Choose%
    \footnote{Such rational functions~$Q(z)$ and~$T(z)$ always exist, as~$R_{n_1,n_2,m}(z)$ may
        have at most finitely many poles, see Theorem~\ref{th:2F1zeros}.}
    a real rational function~$Q(z)$ and a real polynomial~$T(z)$ such
    that~$T(z)\big(R_{n_1,n_2,m}(z)-Q(z)\big)$ is analytic in~$\C\setminus[1,+\infty)$ and
    \begin{equation}\label{eq:asymp_b}
        \lim_{y\to+0}\big|T(x+iy)\big|\cdot\big|R_{n_1,n_2,m}(x+iy)-Q(x+iy)\big|<\infty
        \quad
        \text{for each}\quad
        x\in(1,+\infty).
    \end{equation}
    Let also~$Q(z)$ be analytic near the origin. Take numbers~$M,N\in\N_0$ such that
    \begin{equation}\label{eq:asympinf1}
        \begin{aligned}
        R_{n_1,n_2,m}(z)-Q(z)&=o\left((1-z)^{-M-1}\right)~\text{as}~z\to 1
        \text{~~and}\\
        R_{n_1,n_2,m}(-z)-Q(-z)&=o(z^{N-M})~\text{as}~z\to\infty,
        \end{aligned}
    \end{equation}
    and denote~$d=\deg T$. Then the following representation holds
    \begin{multline}\label{eq:mainrepresentionQN}
        R_{n_1,n_2,m}(z)=Q(z)
        \\
        +
        \sum_{k=0}^{N+d-1} \!\frac{z^k}{(1-z)^MT(z)}
        \sum_{j=0}^k
        \frac{R_{n_1,n_2,m}^{(j)}(0)-Q^{(j)}(0)}{j!}
        \sum_{h=0}^{\min\{k-j,M\}}\!\binom{M}{h}\frac{(-1)^{h}T^{(k-j-h)}(0)}{(k-j-h)!}
        \\
        + \frac{B_{n_1,n_2,m}(a,b,c) z^{N+d}}{(z-1)^MT(z)} \int_1^{\infty}
        \frac{x^{l-\underline{n}-c-N-d}(x-1)^{M+c-a-b-l}T(x)P_{r}(1/x)}{|{}_{2}F_{1}(a,b;c;x)|^{2}(x-z)}dx,
    \end{multline}
    where $r,l,\underline{n}$ and $B(a,b,c)$ are the same as in
    Theorem~\ref{th:2F1ratioboundary} and~$P_r$ is defined in \eqref{eq:Prexplicit}. In the
    case~$N=d=0$, the sum in~$k$ on the right-hand side of~\eqref{eq:mainrepresentionQN} is
    void.
\end{theorem}
As no poles can be produced by the Cauchy-type integral in~\eqref{eq:mainrepresentionQN}, each
pole of~$R_{n_1,n_2,m}(z)$ excluding a possible one at~$z=1$ must be a pole of~$Q(z)$ or a zero
of~$T(z)$ of the same multiplicity. In fact, one can always eliminate one of the
functions~$Q(z),T(z)$ by setting either~$Q(z)\equiv 0$ or~$T(z)\equiv 1$, so that the other
function absorbs the poles of~$R_{n_1,n_2,m}(z)$.

\begin{proof}
    Apply Lemma~\ref{lemma:Schwarz_asympt} with~$f(z)=R_{n_1,n_2,m}(z)-Q(z)$
    and~$q(z)=(1-z)^M T(z)$, then observe that
    \[
        \frac{d^j}{dz^j}(1-z)^M\Big|_{z=0}
        =
        \begin{cases}
            {(-1)^jM!}/{(M-j)!},& \text{if}\quad j\le M;\\
            0,& \text{if}\quad j>M,
        \end{cases}
    \]
    and hence
    \[
    \begin{aligned}
        \frac{q^{(k-j)}(0)}{(k-j)!}
        &=
        \frac{d^{k-j}}{dz^{k-j}}\frac{(1-z)^MT(z)}{(k-j)!}\Big|_{z=0}
        =\frac 1{(k-j)!}\! \sum_{h=0}^{\min\{k-j,M\}}\binom{k-j}{h}\frac{(-1)^{h}M!}{(M-h)!}
        T^{(k-j-h)}(0)
        \\
        &=\sum_{h=0}^{\min\{k-j,M\}}\binom{M}{h}\frac{(-1)^{h}T^{(k-j-h)}(0)}{(k-j-h)!}
        .
    \end{aligned}
    \]
    The only detail we have to take care of is that~$\Im\big(f(x+i0)\big)(1-x)^M T(x)/x^{N+d+1}$
    must be absolutely integrable on~$(1,+\infty)$.

    Due to~\eqref{eq:asymp_b}, it is enough to check the integrability
    of~$\Im\big(f(x+i0)\big)(1-x)^M/x^{N+1}$ as~$x\to 1^+$ and~$x\to +\infty$. For~$x\to+\infty$
    this property holds: the leading term of~$f(x+i0)x^{M-N}=o(1)$ is no worse
    than~$O(1/\log(x))$
    yielding~$\Im\big(f(x+i0)\big)(1-x)^M/x^{N+1}=O\big(x^{-1}|\log(x)|^{-2}\big)$ which is integrable. Details can be found
    in~\cite[Proof of Theorem~2.12]{DK}.

    A similar reasoning applies for~$x\to 1^+$: from~\cite[Lemma~2.5]{DK} one obtains that the
    leading term of~$f(x+i0)(1-x)^M$ is no worse than~$O\big((1-x)^{-1}/\log(x-1)\big)$, which
    implies~$\Im\big(f(x+i0)\big)(1-x)^M/x^{N+1}=O\big((x-1)^{-1}|\log(x-1)|^{-2}\big)$ confirming integrability.
\end{proof}

Let us illustrate Theorem~\ref{th:2F1ratio-gen-repr} by (essentially) repeating the course of
the proof for a particular case in more detail. Given real parameters $a,b,c$, Runckel's
theorem (see Theorem~\ref{th:2F1zeros} below) gives an explicit formula for computing the number
of poles of~$R_{n_1,n_2,m}(z)$ in the cut plane $\C\setminus[1,\infty)$ and on the banks of the
branch cut. As the function ${}_2F_1(a,b;c;z)$ takes complex conjugate values at complex conjugate
points, non-real zeros come in complex conjugate pairs. Denote the zeros in the upper half plane
by $\alpha_1,\ldots,\alpha_k$ and the real zeros by $\beta_{1},\ldots,\beta_{l}$, so that total
number of zeros is $d=2k+l$. Then the function
\[
\begin{aligned}
    \hat{R}_{n_1,n_2,m}(z) &=\frac{T(z)}{z^{d}}R_{n_1,n_2,m}(z),
    \\[3pt]
    \text{where}~~
    T(z)
    &=(z-\alpha_{1})(z-\overline{\alpha}_{1})\cdots(z-\alpha_{k})(z-\overline{\alpha}_{k})(z-\beta_{1})\cdots(z-\beta_{l})
    \\
    &=z^d+\gamma_{d-1}z^{d-1}+\cdots+\gamma_{1}z+\gamma_{0},
\end{aligned}
\]
has the same asymptotics at $z=\infty$ as $R_{n_1,n_2,m}(z)$ and has no poles other then the
pole of order $d$ at $z=0$. Hence, this pole can be removed by subtracting the principal part of
the Laurent series of~$\hat{R}_{n_1,n_2,m}(z)$ at $z=0$. This leads to the definition:
$$
\bar{R}_{n_1,n_2,m}(z)=\hat{R}_{n_1,n_2,m}(z)-L(z),
\quad\text{where}\quad
L(z)=\frac{A_1}{z}+\frac{A_2}{z^2}+\cdots+\frac{A_d}{z^d}.
$$
Then, in view of $R_{n_1,n_2,m}(0)=1$, we will have
\begin{equation*}
\begin{aligned}
    A_{d}{}=&\lim\limits_{z\to0}[\hat{R}_{n_1,n_2,m}(z)z^d]
    =\gamma_0=(-1)^{2k+l}|\alpha_{1}|^2\cdots|\alpha_{k}|^2\beta_{1}\cdots\beta_{l},
    \\
    A_{d-1}{}=&\lim\limits_{z\to0}\frac{d}{dz}[\hat{R}_{n_1,n_2,m}(z)z^d]
    =[T(z)\frac{d}{dz}R_{n_1,n_2,m}(z)]_{|z=0}+[R_{n_1,n_2,m}(z)\frac{d}{dz}T(z)]_{|z=0}
    \\
    &=\gamma_{1}+\gamma_{0}\left(\frac{(a+n_{1})(b+n_{2})}{(c+m)}-\frac{ab}{c}\right),
    \\
    A_{d-j}{}=&\lim\limits_{z\to0}\frac{d^j}{dz^j}[\hat{R}_{n_1,n_2,m}(z)z^d]
    =\sum\limits_{n=0}^{j}\binom{j}{n} [R_{n_1,n_2,m}^{(n)}(z)T^{(j-n)}(z)]_{|z=0}
    \\
    {}=&j!\sum\limits_{n=0}^{j}
    \frac{R_{n_1,n_2,m}^{(n)}(0)}{n!}\gamma_{j-n},\qquad j=2,\ldots,d-1,
\end{aligned}
\end{equation*}
where $\gamma_{j}$ are the coefficients of the polynomial $T(z)$.

There are several ways to compute the $n$-th derivative of the quotient $R_{n_1,n_2,m}(z)$ based
on Fa\`{a} di Bruno's formula. One particularly simple version suggested by Al-Jamal
\cite{AlJamal} reads:
\[
R_{n_1,n_2,m}^{(n)}(z)
=\sum_{k=0}^{n}(-1)^{k}\binom{n+1}{k+1}
\frac{\big({}_2F_1(a+n_1,b+n_2;c+m;z)[{}_2F_1(a,b;c;z)]^k\big)^{(n)}} {[{}_2F_1(a,b;c;z)]^{k+1}}.
\]
This can be combined with Leibnitz's rule \cite[(1.1.1.5)]{Brychkov} to get
\begin{multline*}
\big({}_2F_1(a+n_1,b+n_2;c+m;z)[{}_2F_1(a,b;c;z)]^k\big)^{(n)}
\\
=n!\!\!\sum_{i_0+i_1+\cdots+i_k=n}\!\!
\frac{(a+n_1)_{i_0}(b+n_2)_{i_0}}{i_0!\,(c+m)_{i_0}}{}_2F_1(a+n_1+i_0,b+n_2+i_0;c+m+i_0;z)
\\[-5pt]
\times\prod_{h=1}^k\frac{(a)_{i_h}(b)_{i_h}}{i_h!\,(c)_{i_h}}{}_2F_1(a+i_h,b+i_h;c+i_h;z)
.
\end{multline*}

Next we note that the asymptotic behaviour of $\bar{R}_{n_1,n_2,m}(z)$ in the neighbourhood of
the point $z=1$ is the same as of $R_{n_1,n_2,m}(z)$, namely
$$
|\bar R_{n_1,n_2,m}(z)|
  \sim |R_{n_1,n_2,m}(z)|
  \sim |L_1|\cdot|1-z|^{\eta(a+n_1,b+n_2,c+m)-\eta(a,b,c)}
  \,[\log(1-z)]^{\varepsilon_1}
  ~~~\text{as}~z\to1
$$
with~$\eta$ defined by~\eqref{eq:asymp_1_eta}, ~$\varepsilon_1\in\{-1,0,1\}$ and~$L_1\ne 0$, see Lemma~\ref{lm:asymptotic}. Assume
that~$\eta(a+n_1,b+n_2,c+m)>\eta(a,b,c)-1$, so that a possible singularity at~$z=1$ is
integrable. In the neighbourhood of $z=\infty$, let us assume that for some~$\tau>0$
and~$C\in\R$
$$
R_{n_1,n_2,m}(z)=Q(z)+\frac{C}{\log(z)}\left(1+O\left([\log(z)]^{-1}\right)\right)+O(z^{-\tau}),~~~~~z\to\infty.
$$
If $Q(z)=r_sz^s+\cdots+r_1z+r_0$, then
\begin{align*}
\frac{Q(z)T(z)}{z^{d}}
&=r_sz^s+(r_{s-1}+r_s\gamma_{d-1})z^{s-1}+(r_{s-2}+r_{s-1}\gamma_{d-1}+r_{s}\gamma_{d-2})z^{s-1}+\cdots+r_0\gamma_0z^{-d}
\\
&=\hat{Q}(z)+O(z^{-1}),
\end{align*}
so that
$$
\bar{R}_{n_1,n_2,m}(z)
=\hat{Q}(z)+\frac{C}{\log(z)}\left(1+O\left([\log(z)]^{-1}\right)\right)+O(z^{-\tau}),~~~~~z\to\infty.
$$
As $\bar{R}_{n_1,n_2,m}(z)$ has no singularities in $\C\setminus[1,\infty)$ and is continuous on
the branch cut, we can apply the Schwarz formula (e.g. in the form of
Lemma~\ref{lemma:Schwarz_asympt}) to the difference $\bar{R}_{n_1,n_2,m}(z)-\hat{Q}(z)$. The
coefficients of $\hat{Q}$ and the numbers $A_1,\ldots,A_d$ are real, so the boundary values of
the imaginary parts of $\bar{R}_{n_1,n_2,m}(z)$ and $R_{n_1,n_2,m}(z)$ are related by
$$
\Im\bar{R}_{n_1,n_2,m}(x\pm i0)=\Re\big(T(x)/x^d\big)\cdot\Im{R}_{n_1,n_2,m}(x\pm i0),
$$
where for $x>1$
$$
\Re(T(x))=T(x)=|x-\alpha_1|^2\cdots|x-\alpha_k|^2(x-\beta_1)\cdots(x-\beta_l).
$$
Then the Schwarz formula and Theorem~\ref{th:2F1ratioboundary} applied to~$\bar{R}_{n_1,n_2,m}(z)$ lead to the
following representation
\begin{multline*}
R_{n_1,n_2,m}(z)=\frac{z^d}{T(z)}(\bar{R}_{n_1,n_2,m}(z)+L(z))
\\
=\frac{z^d\hat{Q}(z)+z^dL(z)}{T(z)}+\frac{B_{n_1,n_2,m}(a,b,c)z^d}{T(z)}
\int_1^{\infty}\frac{T(x)x^{l-\underline n-c-d}(x-1)^{c-a-b-l}P_{r}(1/x)}{|{}_{2}F_{1}(a,b;c;x)|^{2}(x-z)}dx,
\end{multline*}
where $l$, $\underline n$, $r$, $B_{n_1,n_2,m}(a,b,c)$ and $P_r$ retain their meanings from
Theorem~\ref{th:2F1ratioboundary}.

\section{Pole-free case}

Given~$\xi\in\R$, let~$\lfloor\xi\rfloor$ be the maximal integer number~$\le\xi$. Note that
if~$\xi$ is non-integer, then $\lfloor -\xi \rfloor=-\lfloor\xi\rfloor-1$. The number of zeros
of the Gauss hypergeometric function may be calculated according to Runckel's
theorem~\cite[Theorem]{Runckel}, which we present here in an extended form:
\begin{theorem}[Runckel]\label{th:2F1zeros}
    Given~$a,b,c\in\R$, where~$-c\notin\N_0$, let~$\xi_1,\dots,\xi_4$ be the numbers~$a,b,c-a,c-b$
    taken in non-decreasing order:
    \[
        \min(a,b,c-a,c-b) = \xi_1\le\xi_2\le\xi_3\le\xi_4 = \max(a,b,c-a,c-b).
    \]
    Denote by~$\nu(a,b,c)$ the number of zeros of~${ }_2F_1(a,b;c;z)$
    in~$\C\setminus[1,+\infty)$, as well as on the upper bank of the branch
    cut~$(1,+\infty)$.

    If~$\{-a,-b,a-c,b-c\}\cap\N_0\ne\varnothing$, then
    \[\textstyle
        \nu(a,b,c)=\xi
        \quad\text{with}\quad
        \xi
        =\min\big(\{-a,-b,a-c,b-c\}\cap\N_0\big)
        =\min\big(\bigcup_{j=1}^4(-\xi_j)\cap\N_0\big);
    \]
    otherwise
    \[
        \nu(a,b,c)=
        \begin{cases}
            0,& \text{if~~} \xi_1>0;\\
            \lfloor -\xi_1\rfloor +\frac{1+S}{2},& \text{if~~} \xi_1<0 \text{ ~and~ }\xi_4>0;\\
            \lfloor -\xi_1\rfloor +\frac{1+S}{2}+S\cdot\lfloor 1-\xi_4\rfloor,&
            \text{if~~} \xi_1<0 \text{ ~and~ }\xi_4<0,
        \end{cases}
    \]
    where~$S=\sign\big(\Gamma(a)\Gamma(b)\Gamma(c-a)\Gamma(c-b)\big)
    =\sign\prod_{j=1}^4\Gamma(\xi_j)$.
\end{theorem}

\begin{proof}
    Consider the degenerate case when~$\bigcup_{j=1}^4(-\xi_j)\cap\N_0\ne\varnothing$.
    If~$\xi=-a$ or~$\xi=-b$, then the function~${ }_2F_1(a,b,c,z)$ reduces to a polynomial of
    degree~$\xi$ that has precisely~$\xi$ zeros in~$\C\setminus\{1\}$: Lemma~\ref{lm:asymptotic}
    shows that~${ }_2F_1(a,b,c,1)\ne 0$. If~$\xi=a-c$ or~$\xi=b-c$, the
    function~${ }_2F_1(a,b,c,z)$ is a (fractional or integer) power of~$(1-z)$ times a
    polynomial of degree~$\xi$ that similarly has precisely~$\xi$ zeros in~$\C\setminus\{1\}$.

    Now, consider the non-degenerate case~$\bigcup_{j=1}^4(-\xi_j)\cap\N_0=\varnothing$. Observe
    that~$\xi_1+\xi_4=c=\xi_2+\xi_3$ in view of~$a+(c-a)=c=b+(c-b)$. When~$\xi_1=a$, we
    automatically obtain~$\xi_4=c-a$, and hence~$c-a\ge b\ge a$. The last condition allows us to
    use~\cite[Theorem]{Runckel}.

    If~$\xi_1=b$, then~$\xi_4=c-b$, as well as~$\nu(a,b,c)=\nu(b,a,c)$. Therefore, application
    of~\cite[Theorem]{Runckel} to ${ }_2F_1(b,a;c;z)$ furnishes the proof. Analogously,
    if~$\xi_1=c-a$ or~$\xi_1=c-b$, then we employ Euler's identity~\eqref{eq:Euler_ID} to see
    that
    \[
        \nu(a,b,c)=\nu(c-a,c-b,c)=\nu(c-b,c-a,c).
    \]
    So, it is enough to apply~\cite[Theorem]{Runckel} to, respectively, ${ }_2F_1(c-a,c-b;c;z)$
    or~${ }_2F_1(c-b,c-a;c;z)$.
\end{proof}

Recall that the banks of the branch cut may only contain zeros in the degenerate
case $\{-a,-b,a-c,b-c\}\cap\N_0\ne\varnothing$, see Lemma~\ref{lemma:no-zeros-branch-cut}. The
function~${ }_2F_1(a,b,c,z)$ is then a polynomial, possibly multiplied by a (fractional or
integer) power of~$(1-z)$. So, Theorem~\ref{th:2F1zeros} mentions ``the upper bank of the branch
cut'' in order to count each zero of that polynomial in~$(1,+\infty)$ exactly one time (the
situation on both banks of the branch cut is the same, as~${ }_2F_1(a,b,c,z)$ is a real
function). This is different to the statement of~\cite[Theorem]{Runckel} that speaks about both
banks of the branch cut together with the branch point~$z=1$. The reason is that~\cite{Runckel}
only touches upon the special non-degenerate case, where this distinction disappears.

One can also count the number of \emph{real} zeros of~${ }_2F_1(a,b;c;z)$ for~$z\in(0,1)$ and
(via Pfaff's transformation~\cite[p.~64, eq.~(22)]{HTF1}) for~$z<0$ by applying the
results of~\cite{Klein,Hurwitz}, see also~\cite{DJ} for the polynomial case.

The following corollary is an extended and improved version of \cite[Theorem~2.1]{DK}.

\begin{corollary}\label{cr:2F1zeros}
    Suppose $c\ne0$. Then ${ }_2F_1(a,b;c;z)$ does not vanish for $z\in\C\setminus[1,+\infty)$
    as well as on the banks of the branch cut if and only if any of the following conditions is
    true:

    \begin{enumerate}[\upshape (I)]
    \item\label{item:Run1} $-1<\min(a,b)\le{c}\le\max(a,b)\le0$;
    \item\label{item:Run2} $-1<\min(a,b)\le0\le\max(a,b)\le{c}$;
    \item\label{item:Run3} $-1<c\le\min(a,b)\le0\le\max(a,b)<c+1$;
    \item\label{item:Run4} $0\le\min(a,b)\le{c}$ ~and~ $\max(a,b)<c+1$;
    \item\label{item:Run5} $a,b,c,c-a,c-b$ are non-integer negative numbers, such
        that~$\lfloor\xi_1\rfloor+1=\lfloor\xi_4\rfloor$
        and~$\lfloor\xi_2\rfloor=\lfloor\xi_3\rfloor$, where~$\xi_1,\dots,\xi_4$ are the
        numbers~$a,b,c-a,c-b$ taken in  non-decreasing order:
        \[
            \min(a,b,c-a,c-b) = \xi_1\le\xi_2\le\xi_3\le\xi_4 = \max(a,b,c-a,c-b);
        \]
    \item\label{item:Run6} $0\in\{a,b,c-a,c-b\}$.
    \end{enumerate}
\end{corollary}

In the notation of condition~\eqref{item:Run5}, one necessarily has~$c-\xi_4=\xi_1$
and~$c-\xi_2=\xi_3$. Indeed: $\xi_1+\xi_4=c=\xi_2+\xi_3$ in view of~$a+(c-a)=c=b+(c-b)$.
Moreover, condition~\eqref{item:Run5} implies~$\xi_1<\xi_2$ and~$\xi_3<\xi_4$: if we had one of
the equalities~$\xi_1=\xi_2$ and~$\xi_3=\xi_4$, we automatically had the other, but the last two
equalities cannot hold simultaneously due to~$\lfloor\xi_1\rfloor+1=\lfloor\xi_4\rfloor$
and~$\lfloor \xi_2\rfloor=\lfloor \xi_3\rfloor$.

\begin{proof}
    On assuming that~$\{-a,-b,a-c,b-c\}\cap\N_0\ne\varnothing$ holds, the condition
    $0\in\{a,b,c-a,c-b\}$ becomes equivalent to that~${ }_2F_1(a,b,c,z)\ne 0$
    for~$z\in\C\setminus[1,+\infty)$ as well as on the banks of the branch cut~$(1,+\infty)$,
    see Theorem~\ref{th:2F1zeros}.

    Let us show now that for~$\{-a,-b,a-c,b-c\}\cap\N_0=\varnothing$ the necessity part holds,
    namely: if ${ }_2F_1(a,b;c;z)\ne 0$ for $z\in\C\setminus[1,+\infty)$ as well as on the banks
    of the branch cut~$(0,1)$, then at least one of the
    conditions~\eqref{item:Run1}--\eqref{item:Run5} is satisfied. By Theorem~\ref{th:2F1zeros},
    either~$\xi_1>0$, or simultaneously~$\xi_1<0$, ~$\prod_{j=1}^4\Gamma(\xi_j)<0$
    and~$(\lfloor \xi_4\rfloor)_--\lfloor \xi_1\rfloor=1$. If~$\xi_1>0$, then~$0<\min(a,b)<c$
    and~$\max(a,b)<c$, so we obtain the condition~\eqref{item:Run4}.

    Let~$-1<\xi_1<0$, then~$\xi_4>0$ and, due to~$\prod_{j=1}^4\Gamma(\xi_j)<0$,
    additionally~$\xi_2\cdot\xi_3>0$. If~$\xi_1=\min(a,b)$, then~$\xi_4=c-\xi_1=c-\min(a,b)$ as
    remarked above, and hence~$(c-\max(a,b))\cdot\max(a,b)>0$. The case~$-1<\xi_1=\min(a,b)<0$
    and~$c-\max(a,b)>0$ yields~$\max(a,b)>0$, and hence that condition~\eqref{item:Run2} holds.
    The case~$-1<\xi_1=\min(a,b)<0$ and~$c-\max(a,b)<0$ yields~$\max(a,b)<0$, and hence that the
    condition~\eqref{item:Run1} holds.

    If~$\xi_1=\min(c-a,c-b)=c-\max(a,b)$, then automatically~$\xi_4=c-\xi_1=\max(a,b)$, and
    hence~$(c-\min(a,b))\cdot\min(a,b)>0$. So, the assumptions~$-1<c-\max(a,b)<0$
    and~$\min(a,b)>0$ yield~$c-\min(a,b)>0$, which leads to condition~\eqref{item:Run4}. In
    turn, the assumptions~$-1<c-\max(a,b)<0$ and~$\min(a,b)<0$ yield~$c-\min(a,b)<0$, and hence
    condition~\eqref{item:Run3} is satisfied.

    Now, let~$\xi_1<-1$. The equality~$(\lfloor \xi_4\rfloor)_--\lfloor \xi_1\rfloor=1$ given by
    Theorem~\ref{th:2F1zeros} then implies
    \[
        \lfloor \xi_4\rfloor=1+\lfloor \xi_1\rfloor<0.
    \]
    In particular,~$\xi_1,\xi_2,\xi_3,\xi_4$ are non-integer negative numbers. Now, the
    additional inequality~$\lfloor \xi_2 \rfloor < \lfloor\xi_3\rfloor$ is impossible, since it
    leads to the contradiction~$\prod_{j=1}^4\Gamma(\xi_j)>0$. Therefore, for~$\xi_1<-1$
    condition~\eqref{item:Run5} is satisfied.

    The sufficiency part for~$\{-a,-b,a-c,b-c\}\cap\N_0=\varnothing$ is
    precisely~\cite[Theorem~2.1]{DK}.
\end{proof}

The next Corollary aims at lifting the integrability condition~$\nu>-1$
of~\cite[Theorem~2.12]{DK}. This corollary turns into a complete generalization
of~\cite[Theorem~2.12]{DK} if the rational function~$Q(z)$ from
Theorem~\ref{th:2F1ratio-gen-repr} is retained (instead of letting~$Q(z)\equiv 0$) and if~$Q(z)$ is
analytic in~$\C\setminus\{1\}$: this rational function then generalizes the
polynomial~$Q_{a,b,c}(z)$ from~\cite[Theorem~2.12]{DK}.

\begin{corollary}\label{cr:2F1ratio-repr}
    Suppose that~$a,b,c\in\R$ and~$n_1,n_2,m\in\Z$, where~$-c,-c-m\notin\N_0$. Let any of the
    conditions~\eqref{item:Run1}--\eqref{item:Run6} in Corollary~\ref{cr:2F1zeros} be satisfied,
    so that ${ }_2F_{1}(a,b;c;z)\ne0$ for all $z$ in~$\C\setminus[1,+\infty)$ and on the banks
    of the branch cut~$(1,+\infty)$. Take numbers~$M,N\in\N_0$ such that
    \begin{equation}\label{eq:asympinf1C}
        R_{n_1,n_2,m}(z)=o\left((1-z)^{-M-1}\right)~\text{as}~z\to 1
        \text{~~and~~}
        R_{n_1,n_2,m}(-z)=o(z^{N-M})~\text{as}~z\to\infty,
    \end{equation}
    or, equivalently, in terms of~$\eta(a,b,c)$ and~$\zeta(a,b,c)$ defined
    in~\eqref{eq:asymp_1_eta}--\eqref{eq:asymp_inf_eta},
    \[
        M>\eta(a,b,c)-\eta(a+n_1,b+n_2,c+m)-1 \text{~~and~~}
        N>M+\zeta(a+n_1,b+n_2,c+m)-\zeta(a,b,c).
    \]
    Then
    \begin{multline}\label{eq:mainrepresentionN}
        R_{n_1,n_2,m}(z)=\sum_{k=0}^{N-1} \frac{z^k}{(1-z)^M} \sum_{j=0}^{\min\{M,k\}}
        (-1)^j\binom{M}{j}\frac{R_{n_1,n_2,m}^{(k-j)}(0)}{(k-j)!}
        \\
        + \frac{z^N}{(z-1)^M} B_{n_1,n_2,m}(a,b,c) \int_1^{\infty}
        \frac{x^{l-\underline{n}-c-N}(x-1)^{M+c-a-b-l}P_{r}(1/x)}{|{}_{2}F_{1}(a,b;c;x)|^{2}(x-z)}dx,
    \end{multline}
    where $r,l,\underline{n}$ and $B_{n_1,n_2,m}(a,b,c)$ are the same as in
    Theorem~\ref{th:2F1ratioboundary} and $P_r$ is defined in \eqref{eq:Prexplicit}.

    In particular, if \eqref{eq:asympinf1C} holds with $N=M=0$ we obtain
    \begin{equation}\label{eq:mainrepresention}
        R_{n_1,n_2,m}(z)=B_{n_1,n_2,m}(a,b,c)\int_1^{\infty}
        \frac{x^{l-\underline{n}-c}(x-1)^{c-a-b-l}P_{r}(1/x)}{|{}_{2}F_{1}(a,b;c;x)|^{2}(x-z)}dx.
    \end{equation}
\end{corollary}

\begin{proof}
    According to~\ref{cr:2F1zeros}, the function~$R_{n_1,n_2,m}(z)$ is analytic
    in~$\C\setminus[1,+\infty)$ and on the banks of the branch cut~$(1,+\infty)$. Therefore, one
    can apply Theorem~\ref{th:2F1ratio-gen-repr} assuming~$Q(z)\equiv 0$ and~$T(z)\equiv 1$.
\end{proof}

\section{Examples and Application}

An interesting application is to plug the integral representations provided by
Theorem~\ref{th:2F1ratio-gen-repr} and Corollary~\ref{cr:2F1ratio-repr} into various
hypergeometric expressions: for instance, into those given in~\cite{Conway}.

By taking derivatives on both sides of formula \cite[eq.~(2.9)]{Conway} and changing $c\to c+1$ we get:
$$
z^c\,\frac{{}_2F_1(a+1,b+1;c+1;z){}_2F_1(a,b;c+1;z)}{\left[{}_2F_1(a,b;c;z)\right]^2}
=\frac{c^2z^{c-1}}{ab}\left(1-R_{0,0,1}(z)\right)-\frac{cz^{c}}{ab}\frac{d}{dz}R_{0,0,1}(z)
$$
On the other hand, according to \cite[Example~8]{DK} we have
\begin{equation}\label{eq:Example8}
R_{0,0,1}(z)=Q_{a,b,c}-\frac{\Gamma(c)\Gamma(c+1)}{\Gamma(a)\Gamma(b)\Gamma(c-a+1)\Gamma(c-b+1)}
\int_0^1\frac{t^{a+b-1}(1-t)^{c-a-b}dt}{(1-zt)|{}_2F_1(a,b;c;1/t)|^2}.
\end{equation}
where
$$
Q_{a,b,c}=\frac{c}{c-\min(a,b)}.
$$
Substituting and simplifying we arrive at the representation
$$
\begin{aligned}
z R_{1,1,1}(z)R_{0,0,1}(z)&=
z\frac{{}_2F_1(a+1,b+1;c+1;z){}_2F_1(a,b;c+1;z)}{\left[{}_2F_1(a,b;c;z)\right]^2}
\\
&=\frac{c^2}{ab}(1-Q_{a,b,c})+B\!\!\int\limits_0^1
\frac{t^{a+b-1}(1-t)^{c-a-b}(c+zt(1-c))dt}{(1-zt)^2|{}_2F_1(a,b;c;1/t)|^2},
\end{aligned}
$$
where
$$
B=\frac{[\Gamma(c+1)]^2}{\Gamma(a+1)\Gamma(b+1)\Gamma(c-a+1)\Gamma(c-b+1)}.
$$
This representation holds under the assumptions of Corollary~\ref{cr:2F1zeros}.

Even more surprising result follows by application of \cite[eq.~(2.10)]{Conway} which after
differentiation and changing $c\to c+1$ reads:
$$
\frac{{}_2F_1(a,b;c-1;z){}_2F_1(a,b;c+1;z)}{\left[{}_2F_1(a,b;c;z)\right]^2}
=\frac{1}{c-1}\left(c-R_{0,0,1}(z)\right)-\frac{z}{c-1}\frac{d}{dz}R_{0,0,1}(z).
$$
Combining this with \eqref{eq:Example8} we obtain the representation containing the so-called
generalized Stieltjes transform of order~$2$ of a positive measure:
\begin{multline*}
    R_{0,0,-1}(z)R_{0,0,1}(z)
    =
    \frac{{}_2F_1(a,b;c-1;z){}_2F_1(a,b;c+1;z)}{\left[{}_2F_1(a,b;c;z)\right]^2}
    \\
    =\frac{c(c-\min(a,b)-1)}{(c-1)(c-\min(a,b))}+\frac{\Gamma(c-1)\Gamma(c+1)}{\Gamma(a)\Gamma(b)\Gamma(c-a+1)\Gamma(c-b+1)}
    \!\!\int\limits_0^1\frac{t^{a+b-1}(1-t)^{c-a-b}dt}{(1-zt)^2|{}_2F_1(a,b;c;1/t)|^2}.
\end{multline*}
This representation also holds under the assumptions of Corollary~\ref{cr:2F1zeros}. In
particular, if parameters are positive and $c>1$, $c-a,c-b>-1$, this representation shows that
the function on the left is monotonically increasing on $(-\infty,1)$. In fact, a
much stronger claim holds: if $c\ge\min(a,b)+1$ and the constant in front of the integral is
positive, then the function $x\to R_{0,0,-1}(-x)R_{0,0,1}(-x)$ is logarithmically completely
monotonic on $[0,\infty)$. This follows from the highly non-trivial inclusion of the Stieltjes
class of order $2$ into the class of logarithmically completely monotonic functions. Details and
history can be found in \cite[Theorem~2.1]{BKP}. Logarithmic complete monotonicity of this
function may be very hard to establish by other means.

It is clear that other formulae from~\cite[Section~2]{Conway} may also be treated in the same way.

\medskip

\noindent\textbf{Example~1.} %
To illustrate Corollary~\ref{cr:2F1ratio-repr}, let us modify the integral expression
for~$R_{1,1,1}(z)$ constructed in~\cite[Example~3]{DK}, so that the result becomes applicable
under milder conditions. Suppose that any of the conditions~\eqref{item:Run1}--\eqref{item:Run5} in
Corollary~\ref{cr:2F1zeros} is satisfied. For simplicity, restrict ourselves to the
non-degenerate case~$a,b,c-a,c-b\notin-\N_0$.

Lemma~\ref{lm:asymptotic} shows that, for~$\tau\coloneqq(c-a-b-1)_--(c-a-b)_-$
and~$\varepsilon_1,\varepsilon_\infty\in\{-1,0,1\}$,
\[
    \begin{aligned}
    R_{1,1,1}(z)
    &=L_1\,(1-z)^{\tau}
    \,[\log(1-z)]^{\varepsilon_1}\big(1+o(1)\big)
    \text{as}~~z\to1
    ,
    \\
    R_{1,1,1}(-z)
    &=L_\infty\,z^{-1}
    \,[\log(z)]^{\varepsilon_\infty}\big(1+o(1)\big)
    \text{as}~~z\to\infty
    .
    \end{aligned}
\]
Observe also that~$\tau\ge (c-a-b)_--1-(c-a-b)_-=-1$, and that~$\tau>-1$ is equivalent
to~$c>a+b$. In~\cite[Example~3]{DK}, we required~$c>a+b$ to make~$R_{1,1,1}(z)$ integrable
near~$z=1$. Now, let us also allow the reverse inequality~$c\le a+b$ that implies~$\tau=-1$.
Corollary~\ref{cr:2F1ratio-repr} with~$N=M=1$ in this case yields
\[
    R_{1,1,1}(z)=\frac{1}{1-z} + \frac{z}{z-1}
    \frac{\Gamma(c)\Gamma(c+1)}{\Gamma(a+1)\Gamma(b+1)\Gamma(c-a)\Gamma(c-b)}
    \int_0^1\frac{t^{a+b}(1-t)^{c-a-b}dt}{(1-zt)|{}_2F_1(a,b;c;1/t)|^2},
\]
which holds true for both~$c>a+b$ and~$c\le a+b$.

\medskip

\noindent\textbf{Example~2.} %
In a similar way, let us modify the integral expression for~$R_{0,0,-1}(z)$ constructed
in~\cite[Example~9]{DK}. Suppose that at least one of the
conditions~\eqref{item:Run1}--\eqref{item:Run5} in Corollary~\ref{cr:2F1zeros} is satisfied, and
that~$a,b,c-a,c-b\notin-\N_0$.
From~\cite[Subsection~2.1]{DK} one can see that
\[
    \begin{aligned}
    R_{0,0,-1}(z)
    &=L_1\,(1-z)^{\tau}
    \,[\log(1-z)]^{\varepsilon_1}\big(1+o(1)\big)
    ~~\text{as}~~z\to1,
    \\
    R_{0,0,-1}(-z)
    &=Q+o(1)~~\text{as}~~z\to\infty
    \end{aligned}
\]
hold for some numbers~$\varepsilon_1\in\{-1,0,1\}$ and~$L_1\ne0$, where
\[
Q=(c-\min(a,b)-1)/(c-1)
\text{~~and~~}
\tau=(c-a-b-1)_--(c-a-b)_-.
\]
Observe that~$\tau\ge -1$, and that the strict inequality~$\tau>-1$ is equivalent to~$c>a+b$.
Now, we remove the assumption~$c>a+b$ of~\cite[Example~9]{DK} introduced there to
make~$R_{0,0,-1}(z)-Q$ integrable near~$z=1$. Theorem~\ref{th:2F1ratio-gen-repr}
with~$N=M=1$ and~$T(z)\equiv 1$ yields
\[
    R_{0,0,-1}(z)=Q+\frac{1-Q}{1-z}
    + \frac{z}{z-1}
    \frac{\Gamma(c)\Gamma(c-1)}{\Gamma(a)\Gamma(b)\Gamma(c-a)\Gamma(c-b)}
    \int_0^1\frac{t^{a+b-1}(1-t)^{c-a-b}dt}{(1-zt)|{}_2F_1(a,b;c;1/t)|^2}
    .
\]
From here, one can also immediately derive an analogous expression for~$R_{0,1,0}(z)$ by
applying the formula~\cite[eq.~12]{Gauss}
\[
R_{0,1,0}(z)=\frac{c-1}{b}R_{0,0,-1}(z)-\frac{c-b-1}{b}.
\]

\medskip

\noindent\textbf{Example~3.} %
For the Gauss ratio $R_{0,1,1}(z)$, Lemma~\ref{lm:2F1identity} and definition \eqref{eq:B-defined}
imply:
$$
B_{0,1,1}P_0(t)\equiv
\frac{\Gamma(c)\Gamma(c+1)}{\Gamma(a)\Gamma(b+1)\Gamma(c-a+1)\Gamma(c-b)}.
$$
So, due to~$\lim\limits_{z\to\infty}R_{0,1,1}(z)=[c(b-a)_+]/[b(c-a)]$
Corollary~\ref{cr:2F1ratio-repr} yields (cf. \cite[Example~1]{DK}):
\begin{equation}\label{eq:R011_integral_repr}
    R_{0,1,1}(z)=\frac{c(b-a)_+}{b(c-a)}
    +\frac{\Gamma(c)\Gamma(c+1)}{\Gamma(a)\Gamma(b+1)\Gamma(c-b)\Gamma(c-a+1)}
    \int\limits_{0}^{1}\frac{t^{a+b-1}(1-t)^{c-a-b}dt}{(1-zt)|{}_2F_1(a,b;c;t^{-1})|^2}.
\end{equation}
In order for this representation to hold we need to assume that any of the conditions
\eqref{item:Run1}-\eqref{item:Run6} of Corollary~\ref{cr:2F1zeros} is satisfied (for
condition~\eqref{item:Run6}, we additionally require~$a\ne c$ to exclude a non-integrable case).
Then verification of the formulae~\eqref{eq:asymp1} and~\eqref{eq:asymp_1_eta} shows
that~$|R_{0,1,1}(z)|$ is integrable near~$z=1$.

Let, for instance,~$0<c<a<c+1$ and $-1<b<0$, so conditions of Corollary~\ref{cr:2F1zeros} are
violated. Theorem~\ref{th:2F1zeros} shows that $R_{0,1,1}(z)$ actually has a unique simple pole
denoted further by~$\beta_1$, ~$\beta_1\ne 1$. Since~$R_{0,1,1}(z)$ is a real function, this
pole is necessarily real. Lemma~\ref{lemma:no-zeros-branch-cut} implies
that~$\beta_1\notin[1,\infty)$. According to Theorem~\ref{th:2F1ratio-gen-repr},
\[
    \begin{aligned}
    R_{0,1,1}(z)
    ={}&\frac{c(b-a)_+}{b(c-a)}+\beta_1\frac{b(c-a)-c(b-a)_+}{b(c-a)(\beta_1-z)}
    \\
    &+\frac{z}{z-\beta_1}
    \frac{\Gamma(c)\Gamma(c+1)}{\Gamma(a)\Gamma(b+1)\Gamma(c-b)\Gamma(c-a+1)}
    \int\limits_{0}^{1}\frac{t^{a+b-1}(1-t\beta_1)(1-t)^{c-a-b}dt}{(1-zt)|{}_2F_1(a,b;c;t^{-1})|^2}
    .
    \end{aligned}
\]
At the origin, the left-hand side in Pfaff's identity~\cite[p.~64, eq.~(22)]{HTF1}
\[
{ }_2F_1(a,c-b;c;z)= (1-z)^{-a}{ }_2F_1\Big(a,b;c;\frac z{z-1}\Big)
\]
has a Taylor expansion with only positive coefficients, so it cannot vanish in~$(0,1)$.
Consequently, the only option is~$\beta_1\in(0,1)$.

Alternatively, the choice~$T(z)\equiv1$ in Theorem~\ref{th:2F1ratio-gen-repr} gives another expression:
\[
    R_{0,1,1}(z)
    =\frac{c(b-a)_+}{b(c-a)} + \frac{A_1}{z-\beta_1}
    +\frac{\Gamma(c)\Gamma(c+1)}{\Gamma(a)\Gamma(b+1)\Gamma(c-b)\Gamma(c-a+1)}
    \int\limits_{0}^{1}\frac{t^{a+b-1}(1-t)^{c-a-b}dt}{(1-zt)|{}_2F_1(a,b;c;t^{-1})|^2}
    ,
\]
where the residue~$A_1$ at~$z=\beta_1$ may be computed through the formula
\[
A_1=\frac {c}{ab}\frac{{ }_2F_1(a,b+1,c+1,\beta_1)}{{ }_2F_1(a+1,b+1,c+1,\beta_1)}.
\]

\end{document}